\documentclass[10pt]{article}
\font\smallit=cmti10

\usepackage{amssymb,latexsym,amsmath,epsfig,amsthm} %% Add other packages as necessary

\usepackage{amsthm}
\usepackage{mathtools}
\usepackage{tikz}
\usepackage[T1]{fontenc}
\usepackage{textcomp}
\usepackage{diagbox}

\usepackage{colortbl}
\usepackage{multirow}
 \usepackage{xcolor}
\usepackage{color}
\usepackage{float}
\usepackage{soul}
\definecolor{labelkey}{rgb}{0,0.08,0.45}
\definecolor{refkey}{rgb}{0,0.6,0.0}

\definecolor{Brown}{rgb}{0.45,0.0,0.05}
\definecolor{lime}{rgb}{0.00,0.8,0.0}
\definecolor{lblue}{rgb}{0.5,0.5,0.99}
\definecolor{lblue}{rgb}{0.8,0.85,1.00}
\definecolor{anotherblue}{rgb}{.8, .8,1}
\definecolor{violet}{rgb}{0.9,0.6,0.9}
\definecolor{greenyellow}{rgb}{0.53,0.99,0.18}

\definecolor{Lyellow}{rgb}{0.87,0.87,0.87}
\definecolor{Lgray}{rgb}{0.92,0.92,0.92}
\definecolor{Mgray}{rgb}{0.5,0.5,0.5}
\definecolor{Gold}{rgb}{0.99,0.84,0.0}

\usepackage{empheq}
\usepackage{siunitx}

\graphicspath{{./figures/}}
\usepackage[colorinlistoftodos,prependcaption,textsize=tiny]{todonotes}
\usepackage[shortlabels]{enumitem}
\usepackage[pdfpagelabels,colorlinks]{hyperref}
\usepackage{mathrsfs}
\definecolor{labelkey}{rgb}{0,0.08,0.45}
\definecolor{refkey}{rgb}{0,0.6,0.0}

\definecolor{Brown}{rgb}{0.45,0.0,0.05}
\definecolor{lime}{rgb}{0.00,0.8,0.0}
\definecolor{lblue}{rgb}{0.5,0.5,0.99}
\definecolor{lblue}{rgb}{0.8,0.85,1.00}
\definecolor{anotherblue}{rgb}{.8, .8,1}
\definecolor{violet}{rgb}{0.9,0.6,0.9}
\definecolor{greenyellow}{rgb}{0.53,0.99,0.18}

\definecolor{Lyellow}{rgb}{0.87,0.87,0.87}
\definecolor{Lgray}{rgb}{0.92,0.92,0.92}
\definecolor{Mgray}{rgb}{0.5,0.5,0.5}
\definecolor{Gold}{rgb}{0.99,0.84,0.0}

\usepackage{siunitx}

\usepackage{amssymb}
\usepackage{graphicx}
\usepackage{amsthm}
\usepackage{mathtools}
\usepackage{tikz}

\usepackage{colortbl}

\graphicspath{{./figures/}}
\usepackage[colorinlistoftodos,prependcaption,textsize=tiny]{todonotes}
\usepackage[shortlabels]{enumitem}
\usepackage{mathrsfs}

\usepackage{color}
\usepackage{graphicx}
\usepackage{caption}
\usepackage{subcaption}

\usepackage[utf8]{inputenc}

\usepackage{pstricks}

\makeatletter

\renewcommand\section{\@startsection {section}{1}{\z@}
{-30pt \@plus -1ex \@minus -.2ex}
{2.3ex \@plus.2ex}
{\normalfont\normalsize\bfseries\boldmath}}

\renewcommand\subsection{\@startsection{subsection}{2}{\z@}
{-3.25ex\@plus -1ex \@minus -.2ex}
{1.5ex \@plus .2ex}
{\normalfont\normalsize\bfseries\boldmath}}

\renewcommand{\@seccntformat}[1]{\csname the#1\endcsname. }

\makeatother

\newtheorem{theorem}{Theorem}

\theoremstyle{definition}
\newtheorem{definition}{Definition}

\newtheorem{remark}{Remark}

\begin{document}

\begin{center}
\uppercase{\bf \boldmath Fixed Points of the Josephus Function via Fractional Base Expansions}
\vskip 20pt
{\bf Yunier Bello-Cruz%\footnote{any footnote here}
}\\
{\smallit Department of Mathematical Sciences, Northern Illinois University, DeKalb, Illinois, USA}\\
{\tt yunierbello@niu.edu}\\ 
\vskip 10pt
{\bf Roy Quintero-Contreras
%\footnote{any footnote here}
}\\
{\smallit Department of Mathematical Sciences, Northern Illinois University, DeKalb, Illinois, USA}\\
{\tt rquinterocontreras@niu.edu}\\ 
\end{center}
\vskip 20pt
%\centerline{\smallit Received: , Revised: , Accepted: , Published: } % We will fill in the dates
\vskip 30pt
\centerline{\bf Abstract}
\noindent In this paper, we investigate properties of the fixed point sequence of the Josephus function $J_3$. First, we establish a connection between this sequence and the Chinese Remainder Theorem. Next, we identify a clear numerical pattern for the digits of two consecutive fixed points when they are written in a non-standard fractional number system in base $3/2$. This result enables us to derive a recursive procedure for determining the digits of their base $3/2$ expansions.

%\vskip 10pt
%\noindent
%\noindent {\bf MSC}: 11B37 \\
%{\bf Keywords}: Josephus function, Fixed points, Fractional base expansion, Chinese Remainder Theorem

%
%\pagestyle{myheadings}
%\markright{\smalltt INTEGERS: 24 (2024)\hfill}
%\thispagestyle{empty}
%\baselineskip=12.875pt
%\vskip 30pt

\section{Introduction}\label{sec:introduction}
This work investigates properties of the Josephus function $J_3$, which arises from the classical Josephus problem, a well-known combinatorial puzzle originating from a narrative by the Jewish-Roman historian Flavius Josephus (see \cite[Book 3, Chapter 8, Part 7]{Fla}). The problem is described as follows: Given a group of $n$ people seated at a round table, numbered consecutively clockwise from $1$ to $n$, begin counting from position $1$, skip two people, and eliminate the third. Continue this process with the remaining participants, always starting with the next person, skipping two, and removing the third, proceeding clockwise until only one person remains. This person is called the survivor.

A natural question is where one should initially sit to avoid elimination. Let $J_3(n)$ denote the position of the survivor. Computing $J_3(n)$ directly for large $n$ can be computationally expensive, so it is useful to identify properties of $J_3$ that can be exploited; see \cite{Bel,Bel2}.

Next, we introduce the notion of extremal points and explain their role in the structure of $J_3$.
\begin{definition}[High extremal points]
A high extremal point is an integer $n_e$ such that
\[
J_3(n_e)\in\{n_e-1,n_e\}.
\]
High extremal points come in two types:
\begin{enumerate}[(a)]
\item If $J_3(n_e)=n_e$, then $n_e$ is a fixed point of $J_3$, denoted by $n_p$.
\item If $J_3(n_e)=n_e-1$, then $n_e$ is a pure high extremal point.
\end{enumerate}
\end{definition}
We denote by $\{n_p^{(\ell)}\}_{\ell\in\mathbb{N}}$ the fixed point sequence, and let $\overline{m}_{\ell}$ be the number of pure high extremal points between $n_p^{(\ell)}$ and $n_p^{(\ell+1)}$ for all $\ell$. These notions not only help characterize the behavior of $J_3$, but also provide insight into its recursive and discrete piecewise linear structure, as illustrated in Figure \ref{f1}.
\begin{figure}[!h]
\centering
\includegraphics[width=1\textwidth]{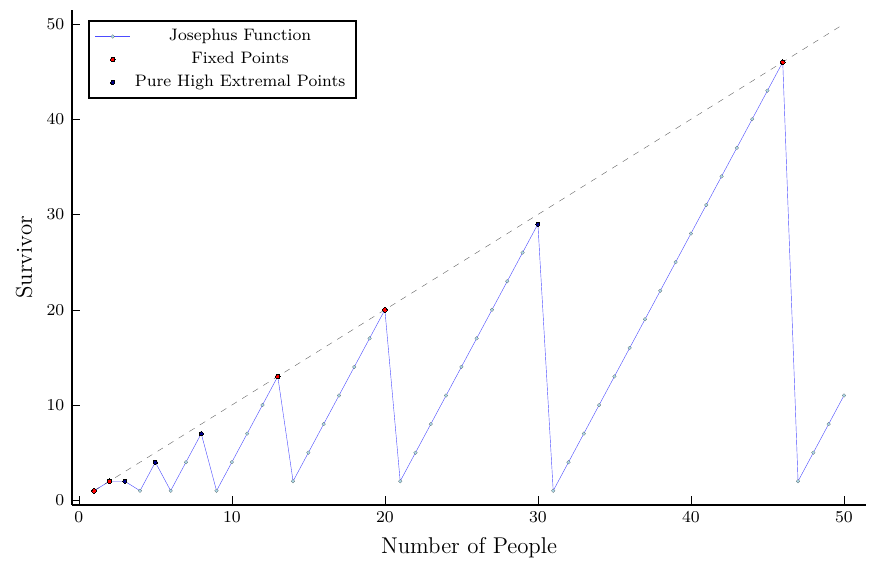}
\caption{Josephus function $J_{3}$ for $n\le 50$.}
\label{f1}
\end{figure}
Next, we outline the main results of this paper, organized into three subsections. We begin by motivating the main goal by recalling the well-known explicit formula for the fixed points of $J_2$. Subsection \ref{sub:fixed_points_of_j_3_and_the_chinese_remainder_theorem} connects the search for fixed points of $J_3$ with the Chinese Remainder Theorem by representing each fixed point as a solution of a system of congruences with pairwise coprime moduli. Subsection \ref{sub:a_modular_expansion_of_any_natural_number_in_base_3_2_} introduces a fractional number system in base $3/2$, following \cite{Aki,Ros}, and studies its relevance to fixed points. Finally, Subsection \ref{sub:characterization_of_n_p_ell_} uses this system to recursively determine the digits in the base $3/2$ expansion of each fixed point of $J_3$.

\section{Main Results and Motivations}\label{:main_results_and_motivations}

In this paper, we aim to derive a recursive formula for the digits of consecutive fixed points of $J_3$ in a suitable base. Such a formula was obtained in full for the special case of $J_2$ in \cite{Gra, Knu}. In this case, the fixed points are given by the explicit formula $n_p^{(\ell)} = 2^\ell - 1$, for all \( \ell \in \mathbb{N} \). Their expansions in base \( 2 \) exhibit a particularly simple recursive pattern. Specifically, if
$$n_p^{(\ell)} = (\hat{b}_k \cdots \hat{b}_0)_{2},$$ $\hat{b}_i\in \{0,1\}$ for all $i$, then the next fixed point satisfies
$$n_p^{(\ell+1)} = (\hat{b}_k \cdots \hat{b}_0\, 1)_{2},$$
that is, the digit \( 1 \) is appended to the right of the base $2$ expansion of $n_p^{(\ell)}$. Moreover, all digits \( \hat{b}_i\) are ones, as shown in the following table:
\begin{table}[!h]
\centering
\begin{tabular}{|c|c|l|}
\hline
\rowcolor[gray]{0.9} 
\( \ell \) & \( n_p^{(\ell)} = 2^\ell - 1 \) & Base $2$ Expansion \\ \hline
1 & 1 & \( (1)_2 \) \\ \hline
2 & 3 & \( (1\underline{\bf 1})_2 \) \\ \hline
3 & 7 & \( (11\underline{\bf 1})_2 \) \\ \hline
4 & 15 & \( (111\underline{\bf 1})_2 \) \\ \hline
5 & 31 & \( (1111\underline{\bf 1})_2 \) \\ \hline
6 & 63 & \( (11111\underline{\bf 1})_2 \) \\ \hline
7 & 127 & \( (111111\underline{\bf 1})_2 \) \\ \hline
8 & 255 & \( (1111111\underline{\bf 1})_2 \) \\ \hline
9 & 511 & \( (11111111\underline{\bf 1})_2 \) \\ \hline
10 & 1023 & \( (111111111\underline{\bf 1})_2 \) \\ \hline
\end{tabular}
\caption{The first ten values of \( n_p^{(\ell)}\) and their base $2$ expansions.}
\label{table:base2-expansion}
\end{table}

This elegant recurrence was the primary motivation for extending this result to $J_3$.

\subsection{Fixed points of $J_3$ and the Chinese Remainder Theorem}\label{sub:fixed_points_of_j_3_and_the_chinese_remainder_theorem}

We begin by recalling the recursive formula for the fixed point sequence \( \{n_p^{(\ell)}\}_{\ell \in \mathbb{N}} \) of the function \( J_3 \), as stated in~\cite[Theorem 2]{Bel}. Starting at \( n_p^{(1)} = 1 \), we compute
\begin{subequations}
\begin{align} \label{fixpointseq}
n_{p}^{(\ell+1)} &= \frac{3^{\overline{m}_{\ell}}(3n_{p}^{(\ell)} + 2) - 2^{\overline{m}_{\ell}}}{2^{\overline{m}_{\ell} + 1}}, \\
\label{m-l-formula}
\overline{m}_{\ell} &= \max \left\{ m \in \mathbb{Z}_+ \mid 2^m \text{ divides } 3n_p^{(\ell)} + 2 \right\},
\end{align}
\end{subequations}
for every \( \ell \in \mathbb{N} \).
From equation~\eqref{fixpointseq}, we can derive, after a few algebraic manipulations, the equivalent identity
\begin{equation}\label{eq-fixpointseq}
2^{\overline{m}_{\ell}}(2n_{p}^{(\ell+1)}+1)=3^{\overline{m}_{\ell}}(3n_{p}^{(\ell)} + 2),
\end{equation}
which tells us that $x_1=n_{p}^{(\ell+1)}$ and $x_2=n_{p}^{(\ell)}$ each satisfy the following system of linear congruences:
\begin{subequations}\label{formula 4}
\begin{empheq}[left=\empheqlbrace]{align}
2x_1 + 1 &\equiv 0 \pmod{3^{\overline{m}_{\ell}}} \label{formula 4-A} \\
3x_2 + 2 &\equiv 0 \pmod{2^{\overline{m}_{\ell}}}. \label{formula 4-B}
\end{empheq}
\end{subequations}
Moreover, if we apply the same argument to $x_1=n_{p}^{(\ell+2)}$ and $x_2=n_{p}^{(\ell+1)}$, then they each satisfy:
\begin{subequations}\label{formula 4'}
\begin{empheq}[left=\empheqlbrace]{align}
2x_1 + 1 &\equiv 0 \pmod{3^{\overline{m}_{\ell+1}}} \label{formula 4-C} \\
3x_2 + 2 &\equiv 0 \pmod{2^{\overline{m}_{\ell+1}}}. \label{formula 4-D}
\end{empheq}
\end{subequations}
Furthermore, since $n_{p}^{(\ell+1)}$ satisfies both \eqref{formula 4-A} and \eqref{formula 4-D}, it must also satisfy the following system of congruences:
\[
\left\{
\begin{aligned}
2x_1 + 1 &\equiv 0 \pmod{3^p} \\
3x_2 + 2 &\equiv 0 \pmod{2^q},
\end{aligned}
\right.
\]
where $p:=\overline{m}_{\ell}$ and $q:=\overline{m}_{\ell +1}$. Note that when $pq=0$, our argument does not apply. From this point on, we assume that $p, q \geq 1$. One can prove by induction on $p$ and $q$ that 
\[
\left\{
\begin{aligned}
2x_1 + 1 &\equiv 0 \!\!\!\! \pmod{3^p} 
&&\!\!\!\!\!\!\!\Longleftrightarrow x_1 \equiv a_1 := \frac{1}{2}(3^p - 1) \!\!\!\! \pmod{3^p} \\[0.36cm]
3x_2 + 2 &\equiv 0 \!\!\!\! \pmod{2^q} 
&&\!\!\!\!\!\!\!\Longleftrightarrow x_2 \equiv a_2 := 
\begin{cases} 
\dfrac{2}{3}(2^q - 1) \!\!\!\! \pmod{2^q}, & \text{if } q \equiv 0 \!\!\!\! \pmod{2} \\[0.36cm]
\dfrac{2}{3}(2^{q-1} - 1) \!\!\!\! \pmod{2^q}, & \text{if } q \equiv 1 \!\!\!\! \pmod{2}.
\end{cases}
\end{aligned}
\right.
\]
Therefore, the fixed point $n_{p}^{(\ell+1)}$ satisfies the following system of linear congruences: 
\begin{equation}\label{formula 6}
\begin{cases}
\ x_1 \equiv a_1 \pmod{3^{p}}  & \\
\ x_2 \equiv a_2 \pmod{2^{q}}. & 
\end{cases}
\end{equation}
Since $3^p$ and $2^q$ are coprime positive integers, i.e., $\gcd(3^p, 2^q)=1$, we can apply the classical form of the Chinese Remainder Theorem (CRT) (see \cite[p. 294]{Rosen}). This guarantees a unique solution $z$ to \eqref{formula 6} modulo $3^p \times 2^q$, given by $$z=a_12^qy+a_23^px,$$
where $x$ and $y$ are a pair of integers satisfying Bézout's identity (see \cite[p. 285]{Rosen}): 
\begin{equation}\label{Bezout-Id}3^px+2^qy=\gcd(3^p, 2^q)=1,\end{equation}
where $\gcd$ denotes the greatest common divisor.
In this way, we uncover a connection between the sequence $\{n_p^{(\ell)}\}_{\ell\in \mathbb{N}}$ and the CRT. 
Moreover, since the explicit expression for $z$ depends on a solution to Bézout's identity, it is natural to examine how $1$ can be expressed as a linear combination of $3^p$ and $2^q$.
\begin{table}[!h]
\centering
\begin{tabular}{|| c | l ||}
\hline
\cellcolor[gray]{0.89} $q$ & \cellcolor[gray]{0.89} Bézout's identity \\
\hline \hline
1 & 
\( 1 = 3^p (1) + 2^1 \left(\tfrac{1}{2}(1 - 3^p)\right), \quad \text{if } p \geq 1 \) \\[0.36cm]
\hline
2 & 
\(
1 =
\begin{cases}
3^p (1) + 2^2 \left( \tfrac{1}{4}(1 - 3^p) \right), & \text{if } p \equiv 0 \pmod{2} \\[0.36cm]
3^p (-1) + 2^2 \left( \tfrac{1}{4}(1 + 3^p) \right), & \text{if } p \equiv 1 \pmod{2}
\end{cases}
\) \\[0.36cm]
\hline
3 & 
\(
1 =
\begin{cases}
3^p (1) + 2^3 \left( \tfrac{1}{8}(1 - 3^p) \right), & \text{if } p \equiv 0 \pmod{2} \\[0.36cm]
3^p (3) + 2^3 \left( \tfrac{1}{8}(1 - 3^{p+1}) \right), & \text{if } p \equiv 1 \pmod{2}
\end{cases}
\) \\[0.36cm]
\hline
4 & 
\(
1 =
\begin{cases}
3^p (1) + 2^4 \left( \tfrac{1}{16}(1 - 3^p) \right), & \text{if } p \equiv 0 \pmod{4} \\[0.36cm]
3^p (-5) + 2^4 \left( \tfrac{1}{16}(1 + 5 \cdot 3^p) \right), & \text{if } p \equiv 1 \pmod{4} \\[0.36cm]
3^p (-7) + 2^4 \left( \tfrac{1}{16}(1 + 7 \cdot 3^p) \right), & \text{if } p \equiv 2 \pmod{4} \\[0.36cm]
3^p (3) + 2^4 \left( \tfrac{1}{16}(1 - 3^{p+1}) \right), & \text{if } p \equiv 3 \pmod{4}
\end{cases}
\) \\[0.36cm]
\hline
5 & $1=\begin{cases} 3^p (1) +2^5 (\frac{1}{32}(1-3^{p})), & \text{ if $p \equiv 0 \pmod {8}$ } \\[0.36cm]
  3^p (11) +2^5 (\frac{1}{32}(1-11\cdot 3^p)), & \text{ if $p \equiv 1 \pmod {8}$ } \\[0.36cm]
  3^p (-7) +2^5 (\frac{1}{32}(1+7 \cdot 3^{p})), & \text{ if $p \equiv 2 \pmod {8}$ } \\[0.36cm]
  3^p (-13) +2^5 (\frac{1}{32}(1+13 \cdot 3^p)), & \text{ if $p \equiv 3 \pmod {8}$ } \\[0.36cm]
  3^p (-15) +2^5(\frac{1}{32}(1+5 \cdot 3^{p+1})), & \text{ if $p \equiv 4 \pmod {8}$ } \\[0.36cm]
  3^p (-5) +2^5(\frac{1}{32}(1+5\cdot 3^p)), & \text{ if $p \equiv 5 \pmod {8}$ } \\[0.36cm]
  3^p (9) +2^5 (\frac{1}{32}(1- 3^{p+2})), & \text{ if $p \equiv 6 \pmod {8}$ } \\[0.36cm]
  3^p (3) +2^5 (\frac{1}{32}(1- 3^{p+1})), & \text{ if $p \equiv 7 \pmod {8}$ }
  \end{cases}$ \\[0.36cm] 
  \hline
\end{tabular}
\caption{Bézout's combinations for any \( p\in\mathbb{N} \) and \( q = 1, 2, 3, 4, 5 \).}
\label{table 1}
\end{table}

Table \ref{table 1} provides Bézout's linear combinations for all $p\in\mathbb{N}$ and $q = 1, 2, 3, 4, 5$ in a compact form.
We invite the interested reader to extend Table \ref{table 1} to other values of $q$.
\begin{remark}
To illustrate the preceding discussion, we present two representative examples from Table \ref{table 2}.

The first example corresponds to \( n_{p}^{(17)} = 3\,986\,218 \) and \( n_{p}^{(18)} = 102\,162\,424 \) given in \cite[Table 1]{Bel2} or Table \ref{table 2}, where $p=7$ (i.e., $\overline{m}_{17}=7$) and $q=1$ (i.e., $\overline{m}_{18}=1$). According to \eqref{formula 6} and Table \ref{table 1}, \( a_1 = 1093 \), \( a_2 = 0 \), and $x=1$, $y=-1093$. Therefore,
$$z= 1093\times 2^1 \times (-1093) + 0 \times 3^7 \times 1=-2389298 \equiv 3280 \pmod{4374}.$$
Thus, $n_{p}^{(18)}=102\,162\,424$ must be congruent to $3280$ modulo $4374$, which is indeed the case since $(102\,162\,424 - 3280)/4374=23356$.

The second example corresponds to \( n_{p}^{(13)} = 46084 \) and \( n_{p}^{(14)} = 103690 \); see Table \ref{table 2} below. In this case, \( p = 1 \) (i.e., \( \overline{m}_{13} = 1 \)) and \( q = 5 \) (i.e., \( \overline{m}_{14} = 5 \)). Here, \( a_1 = 1 \), \( a_2 = 10 \) in \eqref{formula 6}, and \( x = 11 \), \( y = -1 \) for Bézout's identity \eqref{Bezout-Id} since
\[ 3^1 \times 11 + 2^5 \times (-1) = 33 - 32 = 1. \]
So,
\[
z = 1 \times 2^5 \times (-1) + 10 \times 3^1 \times 11 = 298 \equiv 10 \pmod{96}.
\]
Hence, \( n_{p}^{(14)}= 103690 \) must be congruent to $10$ modulo $96$. Indeed, $(103690 - 10)/96 = 1080$.
\end{remark}

\subsection{Expansions in base $3/2$}\label{sub:a_modular_expansion_of_any_natural_number_in_base_3_2_}
We begin this subsection by noting that, by algebraic manipulations of the right-hand side of~\eqref{fixpointseq}, we obtain the following expression:
\begin{equation}\label{formula 7}
n_{p}^{(\ell+1)} = \left(\frac{3}{2}\right)^{\overline{m}_{\ell}+1} n_{p}^{(\ell)} + \left(\frac{3}{2}\right)^{\overline{m}_{\ell}} - \frac{1}{2},
\end{equation}
which provides motivation for studying representations of natural numbers in base $3/2$. Inspired by \eqref{formula 7} and the basic conceptual framework presented in the introduction of \cite[Section 3]{Ros}, we move from base $-3/2$ to base $3/2$, which will allow us to characterize the fixed points of $J_3$ in a much simpler way. Here, we do not refer to the well-known fractional number system in base $3/2$ with digits $0$ or $1$. To illustrate how such a system works, let us represent the natural number $4$ in base $3/2$. A natural approximation is $1000.01001001_{3/2}$, since
$$\left(\frac{3}{2}\right)^3+\left(\frac{3}{2}\right)^{-2}+\left(\frac{3}{2}\right)^{-5}+\left(\frac{3}{2}\right)^{-8}=3.99015012955.$$
Instead of relying on that system, we aim to represent any positive integer $N$ in the form:
\begin{equation}\label{expansion}
 N= \frac{1}{2}\left [d_k \left(\frac{3}{2} \right)^{k}+ d_{k-1} \left(\frac{3}{2} \right)^{k-1} + \cdots + d_1 \left(\frac{3}{2} \right)^1+d_0\right],   
\end{equation}
where $d_i \in \mathcal{D}=\{0, 1, 2\}$ for some $k\in \mathbb{Z}_+$. If $k\geq 1$, we always assume that $d_k \neq 0$. When $N$ has an expansion of the form \eqref{expansion}, we denote
$N=(d_k\cdots d_0)_{3/2}.$
Note that this expansion allows the digit $2$, which is not permitted in the standard base $3/2$ representation. Indeed, the expansion of the number $4$ in the form of \eqref{expansion} is $(212)_{3/2}$, since
\[
\frac{1}{2}\left[2\left(\frac{3}{2}\right)^{2}+1\left(\frac{3}{2}\right)^{1}+2\right]=4.
\]

We now describe the process for computing an expansion of the form \eqref{expansion} via the following four-step algorithm, presented in \cite[p.~61]{Aki}:

\medskip

\noindent \underline{\bf Step 1}: Set $N_0 = N$ and write $2N_0=3N_1 +d_0$ with $d_0 \in \mathcal{D}$ and $N_1 \in \mathbb{Z}_+$. Note that $d_0$ is the unique digit in $\mathcal{D}$ satisfying $d_0 \equiv 2N_0 \pmod{3}$, and thus $N_1$ is also uniquely determined.

\medskip
    
\noindent \underline{\bf Step 2}: Define the non-negative integers $N_1, N_2, \ldots$ recursively by $2N_i=3N_{i+1} +d_i$,
    with $d_i \in \mathcal{D}$ uniquely defined by $d_i \equiv 2N_i \pmod{3}$.

\medskip
    
\noindent \underline{\bf Step 3}: We prove by induction on $i$ that 
    \begin{equation}\label{formula 9}
        N=\left(\frac{3}{2} \right)^{i+1} N_{i+1}+\frac{1}{2}\left [d_i \left(\frac{3}{2} \right)^{i}+ \cdots + d_1 \left(\frac{3}{2} \right)+d_0\right].
    \end{equation}
We verify the base case $i=0$: from {\bf Step 1}, $$N=N_0=\frac{3}{2}N_1+\frac{1}{2}d_0,$$ which confirms that \eqref{formula 9} holds for $i=0$.
Now assume that \eqref{formula 9} holds for $i=k\geq 0$.
From {\bf Step 2}, we have 
$$2N_{k+1}=3N_{k+2} +d_{k+1}.$$ Substituting this into the inductive hypothesis gives:
\begin{align*}
    N&=\left(\frac{3}{2} \right)^{k+1} N_{k+1}+\frac{1}{2}\left [d_k \left(\frac{3}{2} \right)^{k}+ \cdots + d_1 \left(\frac{3}{2} \right)+d_0\right] \\
     &=\left(\frac{3}{2} \right)^{k+1}\left(\frac{3}{2}N_{k+2}+\frac{1}{2}d_{k+1} \right) +\frac{1}{2}\left [d_k \left(\frac{3}{2} \right)^{k}+ \cdots + d_1 \left(\frac{3}{2} \right)+d_0\right] \\    
     &=\left(\frac{3}{2} \right)^{k+2} N_{k+2}+\frac{1}{2}\left [d_{k+1} \left(\frac{3}{2} \right)^{k+1}+ \cdots + d_1 \left(\frac{3}{2} \right)+d_0\right],
\end{align*}
which proves that \eqref{formula 9} holds for $i=k+1$. Thus, by induction, equation \eqref{formula 9} holds for all $i\geq 0$.

\medskip

\noindent\underline{\bf Step 4}: We show that $N_{\hat i} = 0$ for some $\hat i$, ensuring that the algorithm terminates and yields a representation of the form \eqref{expansion}.

From the recurrence relation $N_{i+1} = \frac{2}{3}N_i - \frac{1}{3}d_i \leq \frac{2}{3}N_i$, the sequence $\{N_i\}$ is strictly decreasing whenever $N_i \geq 3$. This guarantees that, after finitely many steps, we reach a value $N_{i_0} \in \{1, 2\}$ for some $i_0 \geq 1$. We consider two cases: (i) if $N_{i_0} = 2$, then $N_{i_0+1} = 1$ and $N_{i_0+2} = 0$, so the expansion terminates at $\hat i=i_0+2$; (ii) if $N_{i_0} = 1$, then $N_{i_0+1} = 0$, hence $\hat i=i_0+1$.

Therefore, for any $N = N_0 \in \mathbb{N}$, the sequence $\{N_i\}$ eventually reaches zero, confirming that $N$ has a finite expansion of the form \eqref{expansion}.
Since each digit $d_i$ is uniquely determined by $d_i \equiv 2N_i \pmod{3}$ and must lie in $\mathcal{D} = \{0,1,2\}$, the expansion is unique.

The following table displays the base $3/2$ expansions of the form \eqref{expansion} for the first twenty fixed points of $J_3$.
\begin{table}[!h]
\centering
\begin{tabular}{||c|r|c|l||}
\hline
\cellcolor[gray]{0.89} $\ell$ & 
\cellcolor[gray]{0.89} $n_p^{(\ell)}$ & 
\cellcolor[gray]{0.89} $\overline{m}_{\ell}$ & 
\cellcolor[gray]{0.89} $3/2$ expansion \\
\hline \hline
1  & 1          & 0 & \( (2)_{3/2} \) \\
2  & 2          & 3 & \( (2\underline{\bf 1})_{3/2} \) \\
3  & 13         & 0 & \( (21\underline{\bf 0112})_{3/2} \) \\
4  & 20         & 1 & \( (210112\underline{\bf 1})_{3/2} \) \\
5  & 46         & 2 & \( (2101121\underline{\bf 02})_{3/2} \) \\
6  & 157        & 0 & \( (210112102\underline{\bf 012})_{3/2} \) \\
7  & 236        & 1 & \( (210112102012\underline{\bf 1})_{3/2} \) \\
8  & 532        & 1 & \( (2101121020121\underline{\bf 02})_{3/2} \) \\
9  & 1198       & 2 & \( (210112102012102\underline{\bf 02})_{3/2} \) \\
10 & 4045       & 0 & \( (21011210201210202\underline{\bf 012})_{3/2} \) \\
11 & 6068       & 1 & \( (21011210201210202012\underline{\bf 1})_{3/2} \) \\
12 & 13654      & 2 & \( (210112102012102020121\underline{\bf 02})_{3/2} \) \\
13 & 46084      & 1 & \( (21011210201210202012102\underline{\bf 012})_{3/2} \) \\
14 & 103690     & 5 & \( (21011210201210202012102012\underline{\bf 02})_{3/2} \) \\
15 & 1181101    & 0 & \( (2101121020121020201210201202\underline{\bf 011112})_{3/2} \) \\
16 & 1771652    & 1 & \( (2101121020121020201210201202011112\underline{\bf 1})_{3/2} \) \\
17 & 3986218    & 7 & \( (21011210201210202012102012020111121\underline{\bf 02})_{3/2} \) \\
18 & 102162424 & 1& \( (2101121020121020201210201202011112102\underline{\bf 01111112})_{3/2} \) \\
19 & 229865455 &0 & \( (210112102012102020121020120201111210201111112\underline{\bf 02})_{3/2} \) \\
20 & 344798183 & 0& \( (21011210201210202012102012020111121020111111202\underline{\bf 1})_{3/2} \) \\
\hline
\end{tabular}
\caption{Values of \( n_p^{(\ell)} \), \( \overline{m}_{\ell} \), and the corresponding base $3/2$ expansions for \( \ell = 1, \ldots, 20 \).}
\label{table 2}
\end{table}
\subsection{Characterization of $\{n_p^{(\ell)}\}_{\ell\in \mathbb{N}}$ through the base $3/2$ number system} \label{sub:characterization_of_n_p_ell_}

Observing the base $3/2$ expansions of the first twenty fixed points of $J_3$ (see Table \ref{table 2}), we are led to believe that there is a natural connection between these expansions and the divisibility condition that defines the sequence $\{\overline{m}_{\ell}\}_{\ell\in \mathbb{N}}$ (see equation \eqref{m-l-formula}). In particular, we have underlined some of the trailing digits of the expansions, as shown in Table \ref{table 2}.

Indeed, Table \ref{table 2} exhibits several notable features:
\begin{enumerate}
\item The base $3/2$ expansion of a fixed point $n_p^{(\ell)}$ appears as the initial segment of the base $3/2$ expansion of the next fixed point $n_p^{(\ell+1)}$.
\item The number of digits in the base $3/2$ expansion of $n_p^{(\ell+1)}$ equals the number of digits in the expansion of $n_p^{(\ell)}$ plus $\overline{m}_{\ell}+1$, where $\overline{m}_{\ell}$ is the number of pure high extremal points between them.
\item The last $\overline{m}_{\ell}+1$ digits in the base $3/2$ expansion of $n_p^{(\ell+1)}$ are determined as follows:
\begin{enumerate}
\item The digit $1$ is appended if $\overline{m}_{\ell}=0$.
\item The digits $02$ are appended if $\overline{m}_{\ell}=1$.
\item The digit strings $012$, $0112$, $01112$, $\ldots$ are appended if $\overline{m}_{\ell}=2,3,4,\ldots$, respectively.
\end{enumerate}
\end{enumerate}
We now state and prove the main result of this work, which confirms and formalizes these observations.

\begin{theorem}
Assume that the base $3/2$ expansion of the fixed point $n_p^{(\ell)}$ of the Josephus function $J_3$ is given by
$n_p^{(\ell)} = (\hat{d}_k \cdots \hat{d}_{0})_{3/2}.$
Then, the following statements hold:
    \item[ {\bf (i)}] If $\overline{m}_{\ell}=0$, then $n_p^{(\ell+1)}=(\hat{d}_k \cdots \hat{d}_{0} \, 1)_{3/2}$.
    \item[ {\bf (ii)}] If $\overline{m}_{\ell}= 1$, then $n_p^{(\ell+1)}=(\hat{d}_k \cdots \hat{d}_{0} \, 0 \, 2)_{3/2}$.
    \item[ {\bf (iii)}] If $\overline{m}_{\ell}\geq 2$, then $n_p^{(\ell+1)}=(\hat{d}_k \cdots \hat{d}_{0} \, 0 \! \underbrace{1 \cdots 1}_{\text{$(\overline{m}_{\ell}-1)$ ones}} \! 2)_{3/2}$.
    \item[ {\bf (iv)}] The total number of digits in the base $3/2$ expansion of $n_p^{(\ell+1)}$ equals the number of digits in the base $3/2$ expansion of $n_p^{(\ell)}$ plus $\overline{m}_{\ell}+1$.
\end{theorem}
\begin{proof}
To prove {\bf (i)}, let $\overline{m}_{\ell}=0$. Then, by \eqref{eq-fixpointseq}, $$2n_p^{(\ell+1)}=3n_p^{(\ell)}+1.$$ So $d_0= 1$, and in the induction formula \eqref{formula 9} with $N=n_p^{(\ell+1)}$ and $N_1=n_p^{(\ell)}$, this proves {\bf (i)}.

To prove {\bf (ii)} and {\bf (iii)}, assume that $\overline{m}_{\ell}\geq 1$. Set $N_0=n_p^{(\ell+1)}$ and define
$$A_i=\dfrac{3n_p^{(\ell)}+2}{2^{\overline{m}_{\ell}-i}}$$
for $i=0, 1, \ldots, \overline{m}_{\ell}$. Note that the definition of $\overline{m}_\ell$ in \eqref{m-l-formula} implies that $A_i\in \mathbb{N}$ for all $i=0, 1, \ldots, \overline{m}_{\ell}$. Then, by \eqref{eq-fixpointseq},
$$2N_0=3(3^{\overline{m}_{\ell}-1}A_0)-1=3(3^{\overline{m}_{\ell}-1}A_0-1)+2.$$
So,
$$d_0={\rm mod}(2N_{0}, 3)=2 \quad \mbox{and} \quad N_1=3^{\overline{m}_{\ell}-1}A_0-1$$
by the recursion \eqref{formula 9}. If $\overline{m}_{\ell}=1$,
$$2N_{1}=2A_{0}-2=A_{1}-2=3n_p^{(\ell)}.$$
Therefore, $d_{1}={\rm mod}(2N_{1},3)=0$, and the digits $2$ and $0$ are appended to the right of the base $3/2$ expansion of $N_{2}=n_p^{(\ell)}$ to represent $n_p^{(\ell+1)}$, establishing {\bf (ii)}.

If $\overline{m}_{\ell}\ge 2$, we repeat the algorithm $\overline{m}_{\ell}-1$ times starting at:

\noindent\underline{Step 1}: We have
$$2N_1=3(3^{\overline{m}_{\ell}-2}A_1)-2=3(3^{\overline{m}_{\ell}-2}A_1-1)+1.$$
Then,
$$d_1={\rm mod}(2N_1,3)=1 \quad \mbox{and} \quad N_2=3^{\overline{m}_{\ell}-2}A_1-1.$$

\noindent Continuing in this way, at \underline{Step $\overline{m}_{\ell}-1$} we obtain
$$2N_{\overline{m}_{\ell}-1}=3A_{\overline{m}_{\ell}-1}-2=3(A_{\overline{m}_{\ell}-1}-1)+1,$$
hence,
$$d_{\overline{m}_{\ell}-1}={\rm mod}(2N_{\overline{m}_{\ell}-1},3)=1 \quad \mbox{and} \quad N_{\overline{m}_{\ell}}=A_{\overline{m}_{\ell}-1}-1.$$
Finally, applying the algorithm once more gives
$$2N_{\overline{m}_{\ell}}=2A_{\overline{m}_{\ell}-1}-2=A_{\overline{m}_{\ell}}-2=3n_p^{(\ell)},$$
so
$$d_{\overline{m}_{\ell}}=0 \quad \mbox{and} \quad N_{\overline{m}_{\ell}+1}=n_p^{(\ell)},$$
proving {\bf (iii)}.

Item {\bf (iv)} follows directly from {\bf (i)}, {\bf (ii)}, and {\bf (iii)}, since each case increases the digit count by exactly $\overline{m}_{\ell}+1$.
\end{proof}

This result is a natural generalization of the explicit formula for the fixed points of $J_2$, discussed at the beginning of this section. Moreover, \cite{Gra, Knu} show that the binary expansion of $n$ can be used to compute $J_2(n)$ via a digit shift. Specifically, if $n=(b_k\cdots b_0)_2$ and $j:=\max\left\{j\in \{k-1,\ldots,0\} \mid b_j\neq 0\right\}$, then $J_2(n)=(b_{j} \cdots b_0 b_k)_{2}$.

Two directions remain open. First, whether the base $3/2$ expansions of $n$ and $n_p^{(\ell+1)}$, together with $\overline{m}_\ell$, can be used to compute $J_3(n)$ directly; equations (9) and (10) in \cite{Bel} may provide a starting point. Second, whether the fixed points of $J_4$ admit a recursive digit formula in base $4/3$; such a representation could reveal structural patterns analogous to those found for $J_3$ in base $3/2$, which in turn generalize the case of $J_2$ in base $2$.

\section{Concluding Remarks}\label{sec:concluding_remarks}

In this paper, we studied the fixed point sequence of the Josephus function $J_3$ and showed how its structure can be better understood through number theory and a fractional base representation. First, we used the Chinese Remainder Theorem to describe each fixed point as the solution of a system of two linear congruences, exploiting the fact that the moduli involved are coprime. Then, we identified a numerical pattern in the base $3/2$ expansions of the fixed points and provided a recursive formula for computing the digits of these expansions.

This connection offers a new and transparent way to describe the fixed points of $J_3$ and opens the door to further analysis using tools from the theory of fractional base representations. The use of base $3/2$ provides a natural perspective for understanding the structure of the fixed points and suggests several directions for future research.

One such direction is to investigate whether the fixed points of $J_4$ admit a recursive digit formula in base $4/3$, potentially revealing patterns analogous to those found for $J_3$ and $J_2$. However, $J_4$ has two distinct types of pure high extremal points, which may require at least two indices to track the number of pure high extremal points between consecutive fixed points.

Another open direction is to determine whether $J_3(n)$ itself can be computed directly from the base $3/2$ expansion of $n$. The expression for $J_3(n)$ given in equations (9) and (10) of \cite{Bel} involves two indices, which complicates the derivation of a recursion for the base $3/2$ expansion of $J_3(n)$.

\vskip 20pt\noindent {\bf Acknowledgements.}
The first author acknowledges the support of the National Science Foundation (NSF) grant \#DMS-2307328. The authors are grateful to the anonymous referees for their valuable comments and suggestions, which improved the quality of the paper.

% \begin{acknowledgment}{Disclosure Statement.}
%No potential conﬂicts of interest were reported by the authors.
% \end{acknowledgment}
% 
%\begin{acknowledgment}{Author Contributions.}
%Yunier Bello-Cruz and Roy Quintero-Contreras contributed equally to this manuscript. Both authors revised the paper critically and approved the version submitted for publication in The American Mathematical Monthly.
%\end{acknowledgment}

\end{document}